\documentclass{article}
\usepackage[english]{babel}
\usepackage[utf8x]{inputenc}
\usepackage{amsmath,amsfonts,amssymb,amsthm}
\usepackage{mathtools}
\usepackage{hyperref}
\usepackage{url}
\usepackage{enumitem}
\usepackage{float}
\usepackage{xcolor}
\usepackage{tikz-cd}
\usepackage{titlesec}
\usepackage{graphicx}
\usepackage[all,cmtip]{xy}
\usepackage{geometry}
\usepackage{svg}
\usepackage{dutchcal} % \mathpzc (alternate mathcal): also lowercase.
\DeclareFontFamily{OT1}{pzc}{}
\DeclareFontShape{OT1}{pzc}{m}{it}{<-> s * [1.10] pzcmi7t}{}
\DeclareMathAlphabet{\mathpzc}{OT1}{pzc}{m}{it}
\def    \C      {{\mathbb C}}
\def    \R      {{\mathbb R}}
\def \Z {{\mathbb Z}}

\newcommand{\grad}{\mathrm{grad}}

\renewcommand{\epsilon}{\varepsilon}

\newtheorem{theorem}{Theorem}[section]
\newtheorem{cor}[theorem]{Corollary}
\newtheorem{lemma}[theorem]{Lemma}

\newtheorem{defi}[theorem]{Definition}

\newtheorem{prob}{Problem}
\newtheorem{ex}{Example}

\theoremstyle{remark}
\newtheorem*{remark}{Remark}

\AtEndDocument{\bigskip{\footnotesize
  \textsc{Universidade Federal do Espírito Santo, Espírito Santo, Brazil} \par  
  \textit{E-mail}:  \texttt{brayan.ferreira@ufes.br}
  }}

\title{Systolic inequalities on the sphere from symplectic embeddings}
\date{}
\author{Brayan Ferreira}

\begin{document}
\maketitle

\begin{abstract}
We use properties of symplectic capacities that were recently defined by Hutchings to obtain upper bounds on the minimal action of Reeb orbits on fiberwise star-shaped hypersurfaces $\Sigma \subset T^*S^2$. In addition, we introduce the notion of a fiberwise $\beta$-balanced hypersurface $\Sigma \subset T^*S^2$ and establish upper bounds for the systole in terms of $\beta$ and geometric data, in the case of Riemannian metrics on $S^2$ satisfying this property. Finally, under the assumption of antipodal symmetry, we provide a non-sharp estimate of how fiberwise balanced a $\delta$-pinched metric is.
\end{abstract}

\tableofcontents

\section{Introduction}
Let $(Y,\lambda)$ be a contact three-dimensional manifold, i.e., $\lambda$ is a $1$-form such that $\lambda \wedge d\lambda \neq 0$. The \emph{Reeb vector field} is defined as the unique vector field $R$ satisfying the equations $d\lambda(R,\cdot) = 0$ and $\lambda(R) = 1$ on $Y$.
The famous Weinstein conjecture says that the Reeb flow on a closed contact manifold always admits a periodic trajectory. The conjecture is proved in dimension $3$ by Taubes \cite{taubes2007seiberg}. Given a periodic Reeb orbit $\gamma \colon \R/T\Z \to Y$, we define its action by
$$\mathcal{A}_\lambda(\gamma) = \int_\gamma \lambda = T.$$
The aim of this paper is to obtain upper bounds on the minimal action
$$\mathcal{A}_{min}(Y,\lambda) := \min\{ \mathcal{A}_\lambda(\gamma) \vert \ \gamma \ \text{periodic Reeb orbit on} \ Y\},$$
for specific contact closed three-dimensional manifolds in the cotangent bundle of the two-dimensional sphere, $T^*S^2$. In what follows, we denote by $S^2 \subset \R^3$ the unit sphere and by $g_0$ the round metric inherited from the Euclidean space.

In \cite{hutchings2022elementary}, Hutchings defined a sequence 

$$ 0 = c_0(X,\omega) \leq c_1(X,\omega) \leq c_2(X,\omega) \leq \ldots \leq +\infty$$
of numerical invariants (symplectic capacities) for any four dimensional symplectic manifold $(X,\omega)$. These numbers are defined as minimax values of energies of suitable pseudoholomorphic curves with point constraints and satisfy nice properties. We recall those we shall use in this work.

\begin{theorem}[{\cite[Theorem 6]{hutchings2022elementary}}]\label{ckprop}
The numbers $c_k$ satisfy the following properties:
\begin{enumerate}
\item (Conformality) $c_k(X,r \omega) = r c_k(X,\omega)$ for any $r>0$.
\item (Monotonicity) If there exists a symplectic embedding $(X_1,\omega_1) \hookrightarrow (X_2,\omega_2)$, then $c_k(X_1,\omega_1) \leq c_k(X_2,\omega_2)$.
\item (Spectrality) If $(X,\omega)$ is a four-dimensional Liouville domain with boundary $Y$, then for each $k$ such that $c_k(X,\omega)<\infty$, there exists an orbit set $\alpha=\{(\gamma_i,m_i)\}$ in $Y$ with $c_k(X,\omega) = \mathcal{A}(\alpha) = \sum_i m_i \mathcal{A}_\lambda(\gamma_i)$.
\item (Ball) The numbers $c_k$ for the \emph{ball of capacity $a$},
$$B(a) = \{x \in \R^4 \mid \pi \Vert x \Vert^2 \leq a\},$$
are given by $c_k(B(a),\omega_0) = da$, where $d$ is the unique nonnegative integer with
$$d^2+d \leq 2k \leq d^2+3d.$$
Here $\omega_0 = dx_i \wedge dy_i$ denotes the standard symplectic form on $\R^4$. In particular, $c_1(B(a),\omega_0) = a$.
\item (Round metric) The numbers $c_k$ for the unit disk cotangent bundle $D^*_{g_0}(1)S^2$ are given by
$$c_k(D^*_{g_0}(1)S^2,\omega_{can}) = \min \{2\pi(m+n) \mid m,n \in \mathbb{N}, \ (m+1)(n+1) \geq k+1\},$$
where $\omega_{can}$ denotes the canonical symplectic form on the cotangent bundle $T^*S^2$. In particular, $c_1(D_{g_0}^*(1)S^2,\omega_{can}) = 2\pi$.
\end{enumerate}
\end{theorem}

As mentioned in \cite[Remark 14]{hutchings2022elementary}, it follows from \cite[Theorem 1.1]{ferreira2021symplectic} and \cite[Lemma 2.3]{oakleyusher} that there exist symplectic embeddings
$$(\mathrm{int} P(2\pi,2\pi),\omega_0) \hookrightarrow (D_{g_0}^*(1)S^2,\omega_{can}) \hookrightarrow (S^2 \times S^2, \sigma(2\pi) \oplus \sigma(2\pi)),$$
where $P(a,b)$ is the symplectic polydisk
$$P(a,b) =\{(z_1,z_2) \in \C^2 \mid \pi \vert z_1 \vert^2 \leq a, \  \vert z_2 \vert^2 \leq b\},$$
$\omega_0$ is the standard symplectic form on $\R^4 = \C^2$ and $\sigma(A)$ is the area form on $S^2$ such that $\int_{S^2} \sigma(A) = A$.

Therefore, by the monotonicity property, we have 
$$c_k(\mathrm{int} P(2\pi,2\pi),\omega_0)\leq c_k(D_{g_0}^*(1)S^2,\omega_{can}) \leq c_k(S^2 \times S^2, \sigma(2\pi) \oplus \sigma(2\pi)),$$
for every $k$. The round metric property in Theorem \ref{ckprop} follows from Hutchings' computations for the polydisks $P(a,b)$ and $S^2 \times S^2$, see \cite[Theorem 17]{hutchings2022elementary}.

Before we state the first result of this work, we recall that given a Riemannian (or more generally, Finsler) metric on a manifold $N$, say $g$, one has Liouville domains associated to it inside the cotangent bundle $T^*N$. In fact, for each $r>0$, we define the \emph{disk cotangent bundle of radius $r$ with respect to $g$} as being the manifold
$$D^*_g(r)N = \{\nu \in T^*N \mid \sqrt{g^*(\nu,\nu)}\leq r\}.$$
Here $g^*$ denotes the dual metric defined by $g$, that is, we have
$$g^*(\nu_1,\nu_2) = g((g^b)^{-1}(\nu_1),(g^b)^{-1}(\nu_2)),$$
where $g^b$ is the vector bundle isomorphism
\begin{align*}
g^b \colon TN &\to T^*N \\ u &\mapsto g(u,\cdot).
\end{align*}
More generally, one can consider a Finsler metric $F\colon TN \to [0,+\infty)$ and define
$$D_F^*(r)N = \{\nu \in T^*N \mid F^*(\nu)\leq r\},$$
where $F^*$ is the (co)-Finsler (or Cartan) metric $F^* \colon T^*N \to \R$ defined by $F^*(\nu) = F(\mathcal{L}^{-1}(\nu))$, where $\mathcal{L}\colon TN \to T^*N$ is the Legendre transform. In this case, it is well known that the canonical symplectic form
$$\omega_{can} = \sum_i dp_i \wedge dq_i$$
restricts to a symplectic structure on $D^*_F(r)N$ such that the boundary
$$\partial D^*_F(r)N = S^*_F(r)N = \{\nu \in T^*N \mid F^*(\nu) = r\}$$
is a contact manifold equipped with the tautological one form
$$\lambda = \sum_i p_idq_i.$$
Moreover, the Reeb flow associated with $\lambda$ coincides with the (co)-geodesic flow for $F$. In particular, the Reeb trajectories are given by $\mathcal{L}(\gamma, \dot{\gamma})$, where $\gamma$ is a geodesic on $N$ such that $F(\gamma) = r$. Also, the action as a Reeb orbit $\mathcal{A}_{\lambda}(\mathcal{L}(\gamma,\dot{\gamma}))$ coincides with the length of the geodesic $L(\gamma) = \int F(\gamma,\dot{\gamma})$. For these facts and more connections between Contact Topology and Finsler Geometry, we recommend \cite{hryniewicz2013introduccao,dorner2017finsler} or also \cite[Appendix B.1.]{abbondandolo2023entropy}.

Our first result follows from the properties listed in Theorem \ref{ckprop}.

\begin{theorem}\label{systolicineq}
Let $(X,\omega)$ be a four-dimensional Liouville domain with boundary being a contact manifold $(Y,\lambda)$. Suppose that there exists a symplectic embedding $(X,\omega) \hookrightarrow (D^*_{g_0}(R)S^2,\omega_{can})$. Then, we have the inequality
$$\mathcal{A}_{min}(Y,\lambda) \leq 2\pi R.$$
\end{theorem}
\begin{proof}
If such an embedding exists, the monotonicity property yields
$$c_1(X_,\omega) \leq c_1(D^*_{g_0}(R)S^2,\omega_{can}).$$
In addition,
$$c_1(D^*(R)S^2,\omega_{can}) = c_1(D^*(1)S^2,R\omega_{can}) = 2\pi R,$$
by the conformality property. The desired inequality follows from the spectrality property, which ensures that $\mathcal{A}_{min}(Y,\lambda) \leq c_1(X,\omega)$.
\end{proof}

On the other hand, one can also study symplectic embeddings in the opposite direction. The monotonicity of $c_k$ yields
\begin{equation}\label{lowerboundc1}
2\pi r = c_1(D_{g_0}^*(r)S^2,\omega_{can}) \leq c_1(X,\omega),
\end{equation}
whenever a symplectic embedding $ (D_{g_0}^*(r)S^2,\omega_{can}) \hookrightarrow (X,\omega)$ exists. Since it is not always true that $c_1(X,\omega)$ coincides with the minimal action of Reeb orbits on the boundary of $X$, we do not recover the opposite direction in Theorem \ref{systolicineq} in general.

\vskip 8pt

\noindent{{\bf Acknowledgments:}}
The author would like to thank Lucas Ambrozio, Urs Frauenfelder, and Felix Schlenk for their helpful conversations and for reading the first draft of this note. Additionally, the author acknowledges the interest and contributions of the participants in the event \emph{Systolic and Diastolic Geometry} (IMPA, March 2025), particularly Alberto Abbondandolo and Johanna Bimmermann, for their input on discussions surrounding Problem \ref{probc_1}.
 
\section{Star-shaped hypersurfaces in $T^*S^2$}
Let $\Sigma$ be a fiberwise star-shaped hypersurface in the cotangent bundle $T^*S^2$, that is,  $\Sigma \subset T^*S^2$ is a smooth compact hypersurface such that each ray emanating from the zero section intersects $\Sigma$ once and transversely. More precisely, for each $q \in S^2$, each positive ray emanating from the origin $0_q \in T^*_qS^2$ intersects $\Sigma_q := \Sigma \cap T^*_qS^2$ in exactly one point and transversely. It is well known that the tautological one-form $\lambda$ restricts to a contact form on $\Sigma$. Therefore, if $X_\Sigma \subset T^*S^2$ is the region enclosed\footnote{That is, for each $q \in S^2$, $X_\Sigma \cap T^*_qS^2$ consists of the segments connecting the origin $0_q \in T^*_qS^2$ and the points in $\Sigma_q$.} by $\Sigma$, we have that $(X_\Sigma, \omega_{can})$ is a four-dimensional Liouville domain with boundary the contact manifold $(\Sigma, \lambda|_\Sigma)$. In this context, Theorem \ref{systolicineq} has the following direct consequence.

\begin{theorem}\label{thm:circumrad}
Let $\Sigma \subset T^*S^2$ be a fiberwise star-shaped hypersurface. Then
$$\mathcal{A}_{min}(\Sigma,\lambda|_\Sigma) \leq 2\pi \mathcal{R}(\Sigma),$$
where $\mathcal{R}(\Sigma)$ is the maximum over all the circumradii:
$$\mathcal{R}(\Sigma) = \max_{\nu \in \Sigma} \sqrt{g_0^*(\nu,\nu)}.$$
\end{theorem}
\begin{proof}
It follows from Theorem \ref{systolicineq} and the fact that the inclusion $X_\Sigma \subset D^*_{g_0}(\mathcal{R}(\Sigma))S^2$ is a symplectic embedding.
\end{proof}

Similarly, because of \eqref{lowerboundc1}, we have
\begin{equation}\label{lowerinr}
c_1(X_\Sigma,\omega_{can}) \geq 2\pi \mathpzc{r}(\Sigma),
\end{equation}
where $\mathpzc{r}(\Sigma)$ is the minimum over all the inradii:
$$\mathpzc{r}(\Sigma) = \min_{\nu \in \Sigma} \sqrt{g_0^*(\nu,\nu)},$$
since $D^*(\mathpzc{r}(\Sigma))_{g_0}S^2 \subset X_\Sigma$.

We note that $X_\Sigma$ generalizes the notion of $D^*_F(r)S^2$. In fact, the case of disk cotangent bundles is exactly given by the fiberwise convex examples, i.e., the cases where $\Sigma_q \subset T^*_qS^2$ bounds a convex subset for every $q \in S^2$. Moreover, in the Riemannian case $F(v) = \sqrt{g(v,v)}$, $\Sigma_q$ is an ellipsoid.

Inspired by the definition of a $\delta$-pinched metric on the sphere and a $\delta$-pinched convex set on $\R^{2n}$, we define the following notion.

\begin{defi}
Let $\beta \in (0,1]$. A fiberwise star-shaped hypersurface $\Sigma \subset T^*S^2$ is fiberwise $\beta$-balanced if
$$\left(\frac{\mathpzc{r}(\Sigma)}{\mathcal{R}(\Sigma)}\right)^2\geq \beta.$$
\end{defi}

For examples, see Section \ref{sec:deltabal}. Note that $\beta=1$ occurs exactly in the case where $\Sigma$ is the sphere cotangent bundle of the round metric, $\Sigma = S^*_{g_0}(R)S^2$, for some $R>0$. In particular, this property measures how far the hypersurface $\Sigma$ is from the round metric in some sense. In \cite{abbondandolo2023entropy}, they define the more general notion of \emph{module of starshapedness} comparing a star-shaped hypersurface with all Finsler metrics in a cotangent bundle $T^*Q$. 

We have the following consequence of Theorem \ref{thm:circumrad} and the volume obstruction for symplectic embeddings.

\begin{theorem}\label{thm:volume}
Let $\Sigma \subset T^*S^2$ be a fiberwise $\beta$-balanced hypersurface. Then
$$\mathcal{A}_{min}(\Sigma,\lambda \vert_\Sigma)^2\leq \frac{\mathrm{Vol}(\Sigma,\lambda \vert_\Sigma)}{2\beta},$$
where $\mathrm{Vol}(\Sigma,\lambda \vert_\Sigma) = \int_\Sigma \lambda_\Sigma \wedge d\lambda_\Sigma = \int_{X_\Sigma} \omega_{can}\wedge \omega_{can} = \mathrm{Vol}(X_\Sigma,\omega_{can})$.
\end{theorem}
\begin{proof}
We note\footnote{We are just multiplying the factor in the fiber direction.} that 
$$D^*_{g_0}(\mathcal{R}(g))S^2 \subset \left(\frac{\mathcal{R}(g)}{\mathpzc{r}(g)}\right)X_\Sigma.$$
Moreover, we have a natural symplectomorphism between $\left(\left(\frac{\mathcal{R}(g)}{\mathpzc{r}(g)}\right)X_\Sigma,\omega_{can}\right)$ and $\left(X_\Sigma,\frac{\mathcal{R}(g)}{\mathpzc{r}(g)} \omega_{can}\right)$. In particular, there exists a symplectic embedding
$$(D^*_{g_0}(\mathcal{R}(g))S^2,\omega_{can}) \hookrightarrow \left(X_\Sigma,\frac{\mathcal{R}(g)}{\mathpzc{r}(g)} \omega_{can}\right).$$
So, the volume obstruction gives
$$8\pi^2(\mathcal{R}(g))^2 \leq \left(\frac{\mathcal{R}(g)}{\mathpzc{r}(g)}\right)^2 \mathrm{Vol}(X_\Sigma,\omega_{can}).$$
From Theorem \ref{thm:circumrad} and the latter inequality, we obtain
\begin{align*}
\mathcal{A}_{min}(\Sigma,\lambda\vert_\Sigma)^2 &\leq 4\pi^2 \mathcal{R}(\Sigma)^2 \\ &\leq \frac{1}{2}\left(\frac{\mathcal{R}(g)}{\mathpzc{r}(g)}\right)^2 \mathrm{Vol}(X_\Sigma,\omega_{can}) \\ &\leq \frac{\mathrm{Vol}(X_\Sigma,\omega_{can})}{2\beta} \\ &= \frac{\mathrm{Vol}(\Sigma,\lambda\vert_\Sigma)}{2\beta}.
\end{align*}
\end{proof}

While the equalities in Theorem \ref{thm:circumrad} and in Theorem \ref{thm:volume} are attained for the case where $\Sigma$ is a sphere cotangent bundle with respect to the round metric on $S^2$, given a specific class or example of $\Sigma \subset T^*S^2$, one can ask whether there exist finer embeddings than the inclusion.

\begin{prob}
Given a Finsler metric $F$ on $S^2$, compute the numbers
\begin{align*}
\inf &\{R \mid \exists \ (D^*_F(1)S^2,\omega_{can}) \hookrightarrow (D^*_{g_0}(R)S^2,\omega_{can})\} \\ \sup &\{r \mid \exists \  (D^*_{g_0}(r)S^2,\omega_{can}) \hookrightarrow (D^*_{F}(1)S^2,\omega_{can})\}.
\end{align*}
\end{prob}

Because of the discussion above, studying this problem may give good estimates on the \emph{systole}, i.e., the length of the shortest closed geodesic for $F$, $L_{min}(F)$. We observe that obtaining sharp systolic inequalities as in \cite{abbondandolo2017systolic,abbondandolo2021sharp} by means of the strategy discussed in this work corresponds to full flexibility of the symplectic embedding problem, namely, finding volume filling symplectic embeddings.

The existence of such a nontrivial embedding (i.e., one that is better than inclusion) is an interesting problem. The \emph{concave into convex toric domain theorem}, due to Cristofaro-Gardiner in \cite{cristofaro2019symplectic}, suggests the existence of nontrivial embeddings in the case of metrics of revolution on $S^2$. Moreover, results due to Lalonde and McDuff \cite[Lemma 1.2 and Theorem 1.3]{lalonde1995local}  suggest that one can \emph{squeeze} and improve the inclusion if the intersection of the image of the inclusion with the boundary $S^*_{g_0}(\mathcal{R}(\Sigma)) S^2$ does not contain a closed characteristic, i.e., a lift of a great circle on $S^2$. If the intersection does contain a closed characteristic, the situation resembles Gromov's nonsqueezing theorem, and then the inclusion is the best one can do; see e.g. \cite[Theorem 5.5]{abbondandolo2019floer}.

\begin{remark}
It is clear that one can repeat the same discussion for any Liouville domain for which $c_1$ is computed. In particular, for the ball of \emph{capacity} $a$, we have $c_1(B(a),\omega_0) = a$. In this case, given a domain $X \subset \R^4$, we have
$$a \leq c_1(X,\omega_0) \leq A,$$
as long as there exists symplectic embeddings $(B(a),\omega_0)  \hookrightarrow (X,\omega_0) \hookrightarrow (B(A),\omega_0)$.
It follows from the remarkable recent work due to Abbondandolo, Edtmair and Kang \cite{abbondandolo2024closed} that 
$$c_1(X,\omega_0) = \mathcal{A}_{min}(\partial X, \lambda_0),$$
whenever $X \subset \R^4$ is a strictly convex domain with smooth boundary $\partial X$ and where
$$\lambda_0 = \frac{1}{2}\sum_{i=1}^2(y_i dx_i - x_i dy_i)$$
is the standard Liouville form on $\R^4$. Therefore, one recovers the fact
$$\pi r(X)^2 \leq \mathcal{A}_{min}(\partial X,\lambda_0) \leq \pi R(X)^2,$$
where $r(X) = \min_{x \in X} \Vert x \Vert$ and $R(X) = \max_{x \in X} \Vert x \Vert$. The lower bound is due to Croke-Weinstein and the upper bound to Ekeland, see \cite[Theorem 4 and Proposition 5]{ekeland2012convexity}.

A similar story holds for $\R P^2$ via the symplectic embeddings
$$(\mathrm{int} B(2\pi), \omega_0) \hookrightarrow (D^*_{g_0}\R P^2,\omega_{can}) \hookrightarrow (\C P^2(2\pi),\omega_{FS}),$$
where we use $g_0$ to indicate the induced round metric on $\R P^2$ and $\C P^2(2\pi)$ indicates the scaled $\C P^2$ so that a line has symplectic area $2\pi$ using the Fubini-Study form $\omega_{FS}$, see \cite[Theorem 1.1]{ferreira2021symplectic} and \cite[Theorem 17]{hutchings2022elementary}.
\end{remark}

\begin{prob}\label{probc_1}
For which Finsler metrics on $S^2$ does $c_1(D^*_F(1)S^2,\omega_{can}) = L_{min}(F)$ hold?
\end{prob}

We note that this equality cannot hold in general. As explained in \cite{ferreira2024elliptic}, given small $\varepsilon>0$, for the \emph{dumbbell metric} $g$ on $S^2$, the systole has length $2\pi \varepsilon$ while $c_1(D^*_g(1)S^2,\omega_{can}) \geq 2 \pi$ whenever the dumbbell contains a hemisphere of the round sphere of constant curvature $K=1$. Moreover, when the metric corresponds to an ellipsoid of revolution $\mathcal{E}(1,1,c) \subset \R^3$, see Example \ref{ellipexample}, it follows from \cite[Theorem 1.2]{ferreira2023gromov} that $c_1(D_g^*(1)S^2,\omega_{can})$ is also greater than the systole for $c > 1$.

\section{Riemannian metrics on $S^2$}

\subsection{Fiberwise $\beta$-balanced metrics}\label{sec:deltabal}
Let $g$ be a Riemannian metric on $S^2$. From now on, we denote by $\mathcal{R}(g):= \mathcal{R}(S^*_g(1)S^2)$ and $\mathpzc{r}(g):= \mathpzc{r}(S^*_g(1)S^2)$ the circumradius and inradius previously defined in the case when $\Sigma = S^*_g(1)S^2$. We say that $g$ is \emph{fiberwise $\beta$-balanced}, $\beta \in (0,1]$, if $S^*_g(1)S^2$ has this property, i.e., if
$$\left(\frac{\mathpzc{r}(g)}{\mathcal{R}(g)}\right)^2\geq \beta.$$
Suppose first that $g$ is conformal to the round one, $g = e^{2\varphi}g_0$ for some smooth function $\varphi \colon S^2 \to \R$. In this case, we have $g^* = e^{-2\varphi}g_0^*$ and, hence,
\begin{align*}
\mathcal{R}(g) &= \max_{g^*(\nu,\nu)=1} \sqrt{g_0^*(\nu,\nu)} \\ &= \max_{e^{-2\varphi}g_0^*(\nu,\nu)=1} \sqrt{g_0^*(\nu,\nu)} \\ &= \max_{p\in S^2} e^{\varphi(p)}.
\end{align*}
In particular, Theorem \ref{thm:circumrad} gives $L_{min}(g) \leq 2\pi e^{\max \varphi}$, for $g = e^{2\varphi}g_0$.

Similarly, we have $\mathpzc{r}(g) = \min_{p\in S^2} e^{\varphi(p)}$, and hence, 
$$\left(\frac{\mathpzc{r}(g)}{\mathcal{R}(g)}\right)^2 = \frac{\min_{p\in S^2}e^{2\varphi(p)}}{\max_{p\in S^2}e^{2\varphi(p)}}.$$
Therefore, a conformal metric $g = e^{2\varphi}g_0$ is fiberwise $\beta$-balanced if, and only if, $\min_{p\in S^2} e^{2\varphi(p)} \geq \beta \max_{p\in S^2} e^{2\varphi(p)}$, or equivalently,
\begin{equation}\label{oscilation}
\mathrm{osc}(\varphi):= \max_{p\in S^2} \varphi(p) - \min_{p \in S^2} \varphi(p) \leq -\frac{1}{2}\ln \beta.
\end{equation}
Thus, if $\varphi$ is $\varepsilon$-small in the $C^0$ topology, $\Vert \varphi \Vert_{C^0}<\varepsilon$, we have
$$\mathrm{osc}(\varphi) \leq 2\Vert \varphi \Vert_{C^0} < 2\varepsilon,$$
and, hence, $e^{2\varphi}g_0$ is fiberwise $(e^{-4\varepsilon})$-balanced.

In particular, if a metric is sufficiently $C^0$-close to the round one, then it is sufficiently fiberwise balanced. Recall that by the Uniformization Theorem, every Riemannian metric $g$ on $S^2$ is isometric to a conformal one $e^{2u}g_0$.
 
From now on, we set
$$K_{min} = \min_{p \in S^2} K_g(p) \quad \text{and} \quad K_{max} = \max_{p \in S^2} K_g(p).$$ Recall that given $\delta \in (0,1]$, a Riemannian metric $g$ is said to be $\delta$-pinched if it is positively curved and $K_{min}/K_{max}\geq \delta$.

A priori, the property of being fiberwise $\beta$-balanced is not related to a pinching condition in the curvature. The next example illustrates that the two properties can be related in some cases.

\begin{ex}\label{ellipexample}
Let $\mathcal{E}(a,b,c) \subset \R^3$ denote the usual ellipsoid defined by the equation
$$\frac{x^2}{a^2}+\frac{y^2}{b^2}+\frac{z^2}{c^2} = 1.$$
Consider the linear map
\begin{align*}
f_{a,b,c}\colon S^2 &\to \mathcal{E}(a,b,c) \\ (x,y,z) &\mapsto (ax,by,cz).
\end{align*}
Define the \emph{ellipsoid metric} $g_{a,b,c}$ on $S^2$ by setting $g_{a,b,c} = f_{a,b,c}^*g_0$, for $a,b,c \in \R_{>0}$, where $g_0$ is the restriction of the Euclidean metric to the ellipsoid. Suppose that $a\leq b\leq c$. It is well known that the minimum and the maximum curvature are given by
$$K_{min} = \frac{a^2}{b^2c^2} \quad \text{and} \quad K_{max} = \frac{c^2}{a^2b^2},$$
respectively, see e.g. \cite[Corollary 3.5.12]{klingenberg1995riemannian}. In particular, we have
$$\frac{K_{min}}{K_{max}} = \left(\frac{a}{c}\right)^4.$$
Moreover, using Lagrange multipliers, one can compute
$$\mathpzc{r}(g_{a,b,c}) = \min\{a,b,c\} = a \quad \text{and} \quad \mathcal{R}(g_{a,b,c}) = \max\{a,b,c\} = c.$$
Therefore,
$$\left(\frac{\mathpzc{r}(g_{a,b,c})}{\mathcal{R}(g_{a,b,c})}\right)^2 = \left(\frac{a}{c}\right)^2.$$
In this case, we conclude that an ellipsoid metric is fiberwise $\sqrt{\delta}$-balanced if, and only if, it is $\delta$-pinched in the classical sense of curvature, which happens exactly when
$$\left(\frac{\min\{a,b,c\}}{\max\{a,b,c\}}\right)^2 \geq \sqrt{\delta}.$$ 
\end{ex}

We shall obtain systolic estimates in terms of geometric data with respect to a Riemannian metric on the sphere.

\subsubsection{Systole and Area}
The comparison of the length of the shortest closed geodesic on a sphere and its area started with Croke \cite{croke1988area}, was then improved by Rotman and Nabutovsky, Sabourau, and the best known upper bound is 
\begin{equation}\label{eq:rotman}
L_{min}(g)^2 \leq 32 \mathrm{Area}(S^2,g),
\end{equation}
due to Rotman \cite{rotman2006length}.
 
Given a Riemannian metric $g$ on $S^2$, we have
$$\mathrm{Vol}(D^*_g(1)S^2,\omega_{can}) = 2\pi \mathrm{Area}(S^2,g),$$
where $\mathrm{Area}(S^2,g) = \int_{S^2}dA_g$. Hence, Theorem \ref{thm:volume} has the following direct consequence. 
\begin{theorem}\label{thm:sysarea}
Let $g$ be a fiberwise $\beta$-balanced Riemannian metric on the sphere $S^2$. Then
$$L_{min}(g)^2 \leq \frac{\pi}{\beta}\mathrm{Area}(S^2,g).$$
\end{theorem}

While the Rotman bound \eqref{eq:rotman} is universal and does not depend on the metric, our bound given by Theorem \ref{thm:sysarea} is not good when $\beta$ is close to zero. Nevertheless, for $\beta \geq \pi/32 \approx 0.098$, our bound is finer than the universal one in \eqref{eq:rotman}.

We note that for spheres of revolution and for sufficiently (curvature) pinched metrics $(\delta >(4+\sqrt{7})/8) \approx 0.83)$, Abbondandolo, Bramham, Hryniewicz and Salomão obtained the sharp systolic inequality
\begin{equation}\label{eq:abhs}
L_{min}^2(g) \leq \pi \mathrm{Area}(S^2,g),
\end{equation}
with equality if, and only if, $g$ is Zoll, using symplectic tools, see \cite{abbondandolo2017systolic} and \cite{abbondandolo2021sharp}. Recently, Vialaret obtained also sharp systolic inequalities for some $S^1$-invariant contact forms on closed three manifolds \cite{vialaret2024sharp}. We note that Theorem \ref{thm:sysarea} also holds for Finsler metrics when using the \emph{Holmes-Thompson} area.

\subsubsection{Systole and first Laplacian eigenvalue}
Given a closed Riemannian manifold $(M,g)$, the first Laplacian eigenvalue $\lambda_1(g)$ is defined as the smallest positive number that satisfies $\Delta_g u + \lambda_1(g) u=0$ for some not identically zero $C^2$-function $u \colon M \to \R$, where $\Delta_g u = \mathrm{div}(\grad_g u)$ denotes the Laplace-Beltrami operator, and $\grad_g u$ is the gradient of the function $u$ with respect to the metric $g$. Using the variational description of $\lambda_1$, Hersch proved the following upper bound
\begin{equation}\label{eq:hersch}
\lambda_1(S^2,g) \leq \frac{8\pi}{\mathrm{Area}(S^2,g)},
\end{equation}
and the equality holds if, and only if, $g$ has constant curvature, see \cite{hersch1970quatre}.
Together with Theorem \ref{thm:sysarea}, we obtain the following consequence.

\begin{cor}\label{cor:lambda}
Let $g$ be a fiberwise $\beta$-balanced Riemannian metric on $S^2$. Then
$$L_{min}(g)^2 \leq \frac{8\pi^2}{\beta \lambda_1(S^2,g)}.$$
\end{cor}

\subsection{Positively curved metrics on $S^2$}

\subsubsection{Systole and Diameter}

Obtaining upper bounds on the systole in terms of the diameter, $D(S^2,g)$, on spheres also starts with Croke \cite{croke1988area} and has an interesting history, see \cite{adelstein2020length}. The best known upper bounds are given by
$$L_{min}(g)\leq 4D(S^2,g),$$
for a general Riemannian metric $g$ on $S^2$, due to Nabutovsky and Rotman and independently Sabourau, and
\begin{equation}\label{eq:ianpal}
L_{min}(g) \leq 3D(S^2,g),
\end{equation}
for non-negatively curved metrics, due to Adelstein and Pallete. Both bounds can be found in the latter reference.

Using Hersch's upper bound \eqref{eq:hersch} and a lower bound due to Zhong-Yang, Calabi and Cao obtained the following interesting estimate \cite{calabi1992simple}. Let $g$ be a Riemannian metric on $S^2$ with nonnegative curvature. Then
\begin{equation}\label{eq:calabicao}
\mathrm{Area}(S^2,g) \leq \frac{8}{\pi}D(S^2,g)^2.
\end{equation}
Putting this together with Theorem \ref{thm:sysarea}, we obtain the following result.

\begin{cor}\label{cor:diam1}
Let $g$ be a fiberwise $\beta$-balanced Riemannian metric on the sphere $S^2$ with nonnegative curvature. Then
$$L_{min}(g) \leq \frac{2\sqrt{2}}{\sqrt{\beta}}D(S^2,g).$$
\end{cor}

Note that our bound is far from good when $g$ is not sufficiently balanced. In fact, this bound in Corollary \ref{cor:diam1} is finer than the more general one due to Adelstein and Pallete in \eqref{eq:ianpal} just for fiberwise $\beta$-balanced metrics with $\beta \geq 8/9$.

Using the sharp inequality $\eqref{eq:abhs}$ and a (curvature) pinched version of \eqref{eq:calabicao}, Adelstein and Pallete obtained the following sharp result: for $\delta$-pinched metrics with $\delta>(4+\sqrt{7})/8\approx 0.83$,
$$L_{min}(g)\leq \frac{2}{\sqrt{\delta}}D(S^2,g)$$
holds with equality if, and only if, the sphere is round.

\subsection{Curvature Pinching vs Fiberwise Balancing}
Finally, we obtain a non-sharp estimate of how a $\delta$-pinched is fiberwise $\beta = \beta(\delta)$-balanced. The motivation here is the following. Let $g= e^{2u}g_0$ be a conformal metric on the sphere. From $\eqref{oscilation}$, we know that $g$ is fiberwise $\beta$-balanced if, and only if, $\mathrm{osc} (u) \leq -\frac{1}{2}\ln \beta$. In this case, we can estimate
\begin{equation}\label{inequalities}
\mathrm{osc}(u) = \mathrm{osc}(u_0) \leq 2\Vert u_0 \Vert_{C^0} \leq 2C_S \Vert u_0 \Vert_{H^2} \leq 2C_SC_P \Vert \Delta_{g_0} u_0 \Vert_{L^2},
\end{equation}
where $u_0 = u - \bar{u}$ is the zero average part of $u$, $C_S$ is a constant coming from the Sobolev embedding $H^2(S^2) \hookrightarrow C^0(S^2)$ and $C_P$ comes from integration by parts and Poincaré Inequality for $u_0$. In Appendix \ref{sec:app}, we shall check that
\begin{align} 
C_S &\leq \frac{1}{\sqrt{4\pi}}\left(\sum_{l=0}^\infty \frac{(2l+1)^2}{1+l(l+1)+l^2(l+1)^2}\right)^{1/2} < \frac{1}{2}\sqrt{\frac{1}{\pi}\left(\frac{\pi^2}{3}+2\right)} \label{ineq:csupper} \\
C_P &\leq \frac{\sqrt{7}}{2},\label{ineq:cpupper}
\end{align}
see Lemma \ref{lemma:cs} and Lemma \ref{lemma:cp}, respectively.

From now on, $\Delta$ and $\nabla$ denote the Laplacian and the gradient with respect to the round metric $g_0$, respectively. 

From Gauss Equation, we have
\begin{align}\label{curveq}
K_g &= e^{-2u}(K_{g_0} - \Delta u) \nonumber \\ &= e^{-2u}(1-\Delta u).
\end{align}
Therefore, $\Delta u = 1 - K_g e^{2u}$ and assuming a $\delta$-pinching condition, say $\delta \leq K_g \leq 1$, one may have control on the $L^2$ norm of $\Delta u = \Delta u_0$.

In fact, such a control is related to the interesting problem of prescribing curvature on spheres, also known as the \emph{Nirenberg problem}, see \cite{moser1973nonlinear,kazdan1974curvature,kazdan1975existence,chang1987prescribing,chang1993nirenberg}. Using Moser-Trudinger type inequalities following Aubin, Chang, Gursky and Yang obtained upper bounds on $\int_{S^2} \Vert \nabla u \Vert^2 dA_{g_0}$ and $\Vert u \Vert_{C^0}$ depending just on the curvature $K_g$ for solutions $u$ of \eqref{curveq}, see \cite{chang1991perturbation,chang1993scalar}. Nevertheless, the constants within the estimates are not explicit and it seems complex to obtain explicit constants.

Inspired by Chang-Yang estimates in \cite{chang1991perturbation}, we obtain explicit constants using the following refinement of Onofri's inequality under the antipodal symmetry hypothesis due to Osgood, Philips and Sarnak.

\begin{lemma}{\cite[Corollary 2.2]{osgood1988extremals}}\label{onofrisarnak}
Let $u \in W^{1,2}(S^2)$ be a mean value zero function such that $u(-q)=u(q)$, for every $q \in S^2$. Then
$$\ln\left(\int_{S^2}e^{u} \ d\sigma_0\right) \leq \frac{1}{8} \int_{S^2} \Vert \nabla u \Vert^2 \ d\sigma_0,$$
where $d\sigma_0 = \frac{1}{4\pi}dA{g_0}$ is the probability measure induced by the round metric on $S^2$.
\end{lemma}

From this, we obtain a control on the $L^2$ norm of the gradient of a solution $u$ of equation \eqref{curveq}.

\begin{lemma}\label{ineq:l2gradient}
Let $u$ be a smooth function on the two sphere with zero average. If $u$ is a solution to equation \eqref{curveq} with $K_g>0$ and $u(q) = u(-q)$ for all $q \in S^2$, then
\begin{align*}
\int_{S^2} \Vert \nabla u \Vert^2 \ d\sigma_0 &\leq \frac{1-K_{min}e^{-2}}{2K_{min}e^{-2}}\left(\ln\left(\frac{K_{max}}{1-K_{min}e^{-2}
}\right)\right) \\ &< \frac{1}{2K_{min}e^{-2}}\ln(K_{max}) + \frac{1}{2}.
\end{align*}
In particular, if $K_g$ is $\delta$-pinched, we have
$$\int_{S^2} \Vert \nabla u \Vert^2 \ d\sigma_0 < \frac{1}{2}\left(\frac{e}{\delta} + 1 \right).$$
\end{lemma}

We recall that in \cite{moser1973nonlinear}, Moser proved that if a positive function $K$ is antipodally symmetric, i.e., $K(x)=K(-x)$ on $S^2$, then equation \eqref{curveq} admits a solution $u$ with the same symmetry.

Since we shall use the Green's function for the Laplacian on the sphere in our estimates, we recall its properties.

\begin{theorem}[{\cite[Theorem 4.13]{aubin1998some}}]\label{greenfunc}
Let $M$ be a $n$-dimensional closed Riemannian manifold. There exists a smooth function $G$ defined on $M \times M$ minus the diagonal with the following properties:
\begin{enumerate}
\item For every $\varphi \in C^2(M)$,
$$\varphi(p) = \frac{1}{\mathrm{Vol}(M)}\int_{q \in M} \varphi(q) \ d\mathrm{Vol}(q) - \int_{q\in M} G(p,q) \Delta \varphi(q) \ d\mathrm{Vol}(q).$$
\item There exists a constant $k$ such that, for every $p\neq q$, 
\begin{align*}
\vert G(p,q) \vert &\leq k(1+\vert \ln d(p,q) \vert), \ \text{for} \ n=2 \\ \vert G(p,q) \vert &\leq k d(p,q)^{2-n}, \ \text{for} \ n>2, \\ \Vert \nabla_q G(p,q) \Vert &\leq kd(p,q)^{1-n}, \\ \Vert \nabla^2_q G(p,q)\Vert &\leq kd(p,q)^{-n}.
\end{align*}
\item There exists a constant $A$ such that $G(p,q) \geq A$. Since the Green's function is defined up to a constant, we can choose the Green's function everywhere positive.
\item $\int_{q \in M} G(p,q) \ d\mathrm{Vol}(q)$ is constant, and hence, we can choose the Green's function so that its integral equals zero.
\item $G(p,q) = G(q,p)$ for $p \neq q$.
\end{enumerate}
\end{theorem}

In the case of the round two-sphere, one can explicitly compute the Green's function for the Laplacian
$$G(p,q) = -\frac{1}{2\pi}\ln (\Vert p - q \Vert_E) + C,$$
for $p,q \in S^2 \times S^2 \subset \R^3 \times \R^3$, where $\Vert \cdot \Vert_E$ denotes the Euclidean norm in $\R^3$, see e.g. \cite[Appendix A.1.]{beltran2019discrete}. We choose $C =\frac{1}{4\pi}(2\ln 2 - 1)$, yielding $\int_{q \in M} G(p,q) \ d\mathrm{Vol}(q) = 0$. 

Indeed, for $p= (0,0,1) \in S^2$, we can write $\Vert p - q \Vert_E = 2\sin(\theta/2)$ in spherical coordinates, where $\theta \in [0,\pi]$ is the polar angle between the radial line and the $z$-axis. Therefore, we compute
\begin{align*}
\int_{q \in S^2} (G(p,q) - C) \ dA_{g_0}(q) &= \int_0^{2\pi}\int_{0}^\pi -\frac{1}{2\pi}\ln\left(2\sin\left(\frac{\theta}{2}\right)\right)\sin \theta \ d\theta d\phi \\ &= -\int_0^\pi \ln\left(2\sin\left(\frac{\theta}{2}\right)\right)\sin \theta \ d\theta \\ &= -4 \int_0^{\pi/2} \ln(2\sin(\gamma))\sin\gamma \cos \gamma \ d\gamma \\ &= -4 \int_0^1 \ln(2v) v \ dv \\ &= -4\left(\int_0^1 \ln(2)v \ dv + \int_0^1 \ln(v) v\ dv\right) \\ &= -2 \ln 2 + 1,
\end{align*}
where we substitute $\theta = 2\gamma$ and $\sin\gamma = v$.

\begin{lemma}\label{minofu}
For a solution $u \in C^2(S^2)$ to equation \eqref{curveq}, the following upper bound holds
$$u(p) \geq \bar{u} - 1 \quad \text{for all} \ p \in S^2,$$
where $\bar{u} = \int u \ d\sigma_0$ denotes the average of $u$.
\end{lemma}
\begin{proof}
From Theorem \ref{greenfunc}, we can write
$$u(p) = \bar{u} - \int_{S^2} G(p,q)\Delta u(q) \ dA_{g_0}(q).$$
Then, using that $u$ solves equation \eqref{curveq} and $G(p,q) = -\frac{1}{2\pi}\ln(\Vert p-q \Vert_E) + \frac{1}{4\pi}(2\ln 2 -1)$,
\begin{align*}
u(p) &= \bar{u} - \int_{q\in S^2} G(p,q) (1-K_ge^{2u} )\ dA_{g_0}(q) \\ &= \bar{u} + \int_{q\in S^2} G(p,q)K_ge^{2u} \ dA_{g_0}(q) \\ &\geq \bar{u} + \left(-\frac{1}{2\pi}\ln 2 + \frac{1}{4\pi}(2\ln 2 -1)\right) \int_{S^2} K_ge^{2u} \ dA_{g_0} \\ &= \bar{u} -1,
\end{align*}
since $\int_{S^2} K_ge^{2u} \ dA_{g_0} = \int_{S^2} K_g \ dA_g = 4\pi$ from Gauss-Bonnet Theorem.
\end{proof}

Now we are ready to prove Lemma \ref{ineq:l2gradient}.

\begin{proof}[Proof of Lemma \ref{ineq:l2gradient}]
If $u$ is constant equal zero, the result follows easily. We assume $u$ nonconstant. From $K_ge^{2u} = 1-\Delta_{g_0}u$, we obtain
\begin{equation}
 2 \int_{S^2} \Vert \nabla u \Vert^2 \ d\sigma_0 + 2 \int_{S^2} u \ d\sigma_0 = 2\int_{S^2}K_g e^{2u}u \ d\sigma_0
\end{equation}
multiplying both sides by $2u$ and integrating over $S^2$. Note that $\int_{S^2} K_ge^{2u} \ d\sigma_0 = \frac{1}{4\pi} \int_{S^2} K_g dA_g = 1$. Then, we can use Jensen's inequality, the fact that $u$ has zero average and Lemma \ref{onofrisarnak} to estimate
\begin{align*}
2 \int_{S^2} \Vert \nabla u \Vert^2 \ d\sigma_0 &=  2\int_{S^2}K_g e^{2u}u \ d\sigma_0 \\ &= 2\int_{S^2}(K_g e^{2u}-m)u \ d\sigma_0 \\ &= (1-m)\int_{S^2}\left(\frac{K_ge^{2u}-m}{1-m}\right)2u \ d\sigma_0 \\ &\leq (1-m) \ln \left( \int_{S^2} \left(\frac{K_ge^{2u}-m}{1-m}\right) e^{2u} \ d\sigma_0\right) \\ &\leq (1-m)\left( \ln(K_{max}) - \ln(1-m)+ \ln\left(\int_{S^2}e^{4u} \ d\sigma_0\right)\right) \\ &\leq (1-m)\left(\ln(K_{max}) - \ln(1-m)+ 2 \int_{S^2} \Vert \nabla u \Vert^2 \ d\sigma_0\right),
\end{align*}
where $m = \min_{S^2} K_ge^{2u}>0$ and $m<1$ follows from Gauss-Bonnet. Thus, we get
\begin{equation}\label{ineqgrad}
\int_{S^2} \Vert \nabla u \Vert^2 \ d\sigma_0 \leq \frac{1-m}{2m}\left(\ln (K_{max}) - \ln(1-m)\right)
\end{equation}

By Lemma \ref{minofu}, we have $\min_{S^2} u \geq -1$ and then, $m \geq K_g e^{-2} \geq K_{min}e^{-2}>0$. Therefore,
$$\frac{1-m}{2m} \leq \frac{1-K_{min}e^{-2}}{2K_{min}e^{-2}} \quad \text{and} \quad -\frac{1-m}{2m}\ln(1-m) \leq -\frac{1-K_{min}e^{-2}}{2K_{min}e^{-2}}\ln(1-K_{min}e^{-2}) < \frac{1}{2}$$
hold since the real functions $x \mapsto \frac{1-x}{2x}$ and $x\mapsto -\frac{1-x}{2x} \ln(1-x)$ are decreasing for $x>0$, and $\lim_{x\to {0_+}} -\frac{1-x}{2x}\ln(1-x) = 1/2$. Putting these together with \eqref{ineqgrad}, we obtain
\begin{align*}
\int_{S^2} \Vert \nabla u \Vert^2 \ d\sigma_0 &\leq \frac{1-K_{min}e^{-2}}{2K_{min}e^{-2}}\left(\ln\left(\frac{K_{max}}{1-K_{min}e^{-2}
}\right)\right) \\ &< \frac{1}{2K_{min}e^{-2}}\ln(K_{max}) + \frac{1}{2},
\end{align*}
as desired. If $K_g$ is $\delta$-pinched, we have $K_{min} \geq \delta K_{max}$. In this case, we get
\begin{align*}
\int_{S^2} \Vert \nabla u \Vert^2 \ d\sigma_0 &< \frac{1}{2K_{min}e^{-2}}\ln(K_{max}) + \frac{1}{2} \\ &\leq \frac{1}{2}\left(\frac{e^2}{\delta K_{max}}\ln(K_{max}) + 1\right) \\ &\leq \frac{1}{2}\left(\frac{e^2}{\delta e} + 1 \right) \\ &\leq \frac{1}{2}\left(\frac{e}{\delta} + 1 \right),
\end{align*}
since the real function $x \mapsto \frac{\ln x}{x}$ attains its global maximum in $x=e$.
\end{proof}

Finally, we obtain our final estimate.

\begin{theorem}
Let $g$ be a $\delta$-pinched Riemannian metric which is antipodally symmetric, i.e., $a^*g = g$ for the antipodal map $a(q)=-q, \ q \in S^2$. Then $g$ is fiberwise $\beta$-balanced with
$$\beta > \exp\left({-2\left(\sqrt{7\left(\frac{\pi^2}{3}+2\right)}\sqrt{\left(\frac{e^{(e/\delta+1)}}{\delta^2}-1\right)}\right)}\right).$$
\end{theorem}

\begin{proof}
Since $g$ is $\delta$-pinched, we have $K_g>0$ and $K_{min}\geq \delta K_{max}$. By the Uniformization Theorem, $g$ is isometric to a conformal metric $e^{2u}g_0$, for some smooth function $u\colon S^2 \to \R$. Since $g$ is antipodally symmetric, we can assume $u(q)=u(-q)$ and without loss, we scale $g$ such that $\bar{u} = 0$. We write $K_{e^{2u}g_0} = K_g$, omitting the uniformization isometry.

In this case, $u$ solves the Gauss equation \eqref{curveq}. In particular, we have
$$\Delta u = 1 - K_ge^{2u}.$$
Squaring both sides and integrating over $S^2$, we obtain
\begin{align}\label{ineq:l2lapl}
\Vert \Delta u \Vert_{L^2}^2 = \int_{S^2} \vert \Delta u \vert^2 \ dA_{g_0} &= 4\pi -2\int_{S^2} K_ge^{2u} \ dA_{g_0} + \int_{S^2} K_g^2e^{4u} \ dA_{g_0} \nonumber \\ &= -4\pi + \int_{S^2} K_g^2e^{4u} \ dA_{g_0} \nonumber \\ &\leq -4\pi + K_{max}^2\int_{S^2}e^{4u} \ dA_{g_0},
\end{align}
where we use again $\int_{S^2} K_ge^{2u} \ dA_{g_0} = 4\pi$. From Lemma \ref{onofrisarnak}, we get
$$\int_{S^2} e^{4u} \ dA_{g_0} \leq 4\pi e^{2\int_{S^2}\Vert \nabla u \Vert^2 \ d\sigma_0},$$
and using Lemma \ref{ineq:l2gradient},
\begin{align*}
e^{2\int_{S^2}\Vert \nabla u \Vert^2 \ d\sigma_0} &< e^{2\left(1/2\left(e/\delta+1\right)\right)} \\ &= e^{\left(e/\delta+1\right)}.
\end{align*}
Incorporating these with \eqref{ineq:l2lapl}, we get
$$\Vert \Delta u \Vert_{L^2}^2 < 4\pi \left(K^2_{max}e^{(e/\delta+1)} -1 \right)\leq 4\pi \left(\frac{e^{(e/\delta+1)}}{\delta^2}-1\right),$$
where we use the pinching condition, Gauss-Bonnet Theorem and Jensen's inequality to obtain $K_{max} \leq 1/\delta K_{min} \leq 1/\delta \frac{4\pi}{\mathrm{Area}(S^2,g)} \leq 1/\delta$. Now combining the latter inequality with \eqref{inequalities}, we obtain
\begin{align}\label{oscilation2}
\mathrm{osc}(u) &\leq 2C_SC_P \Vert \Delta u_0 \Vert_{L^2} \nonumber \\ &= 2C_SC_P \Vert \Delta u  \Vert_{L^2} \nonumber \\ &< 2C_SC_P 2\sqrt{\pi} \sqrt{\left(\frac{e^{(e/\delta+1)}}{\delta^2}-1\right)} \nonumber \\ &< \sqrt{7\left(\frac{\pi^2}{3}+2\right)}\sqrt{\left(\frac{e^{(e/\delta+1)}}{\delta^2}-1\right)},
\end{align}
using the upper bounds \eqref{ineq:csupper} and \eqref{ineq:cpupper}. 

We have seen in \eqref{oscilation} that $g$ is fiberwise $\beta$-balanced if, and only if $\mathrm{osc}(u)\leq -1/2\ln \beta$. From \eqref{oscilation2}, we get that $g$ is fiberwise $\beta$-balanced, where
$$\beta > \exp\left({-2\left(\sqrt{7\left(\frac{\pi^2}{3}+2\right)}\sqrt{\left(\frac{e^{(e/\delta+1)}}{\delta^2}-1\right)}\right)}\right).$$
\end{proof}

As a consequence, one derives from Theorem \ref{thm:sysarea}, Corollary \ref{cor:lambda} and Corollary \ref{cor:diam1}, non-sharp systolic inequalities for antipodally symmetric metrics which are $\delta$-pinched.

\appendix
\section{Appendix: Estimating Constants}\label{sec:app}
In this section, we estimate the constants $C_S$ and $C_P$ appearing in $\eqref{inequalities}$. For $C_S$, we estimate a constant for the Sobolev embedding $H^2(S^2) \hookrightarrow C^0(S^2)$ using the Laplace-Fourier series, i.e., the spherical harmonics expansion.

Recall that the Laplacian's spherical harmonics $Y_{lm}$ are the restriction of the harmonic homogeneous polynomials of degree $l$ to $S^2$. These are the eigenfunctions for the Laplacian with respect to the round sphere:
$$\Delta Y_{lm} = -l(l+1)Y_{lm}, \quad l \in \mathbb{Z}_{\geq 0}, \ -l \leq m \leq l.$$
Moreover, the set $\{Y_{lm}\}$, $l \in \mathbb{Z}_{\geq 0}$ and $-l \leq m \leq l$, form a complete set of orthogonal functions in the Hilbert space $H^2(S^2) = W^{2,2}(S^2)$ consisting of the completion of $C^2(S^2)$ with respect to the norm
$$\Vert f \Vert_{H^2} = \left( \Vert f \Vert_{L^2}^2 + \Vert \nabla f \Vert_{L^2}^2 + \Vert \Delta f \Vert_{L^2}^2\right)^{1/2}.$$
Further, given $u \in H^2(S^2)$, the Laplace-Fourier series
$$u = \sum_{l=0}^\infty \sum_{m=-l}^l a_{lm}Y_{lm},$$
where $a_{lm} = \int_{S^2} uY_{lm} \ dA_{g_0}$, converges uniformly to $u$. We recommend \cite{stein1971introduction,garret} for these and further details on spherical harmonics. We adopt the normalization $\Vert Y_{lm} \Vert_{L^2} = 1$. In this case, it follows that
$$\Vert Y_{lm} \Vert_{C^0} = \sqrt{\frac{2l+1}{4\pi}}, \quad l \in \mathbb{Z}_{\geq 0}, \ -l \leq m \leq l.$$

\begin{lemma}\label{lemma:cs}
Let $u \in H^2(S^2) = W^{2,2}(S^2)$, then
$$\Vert u \Vert_{C^0} \leq \frac{\Vert u \Vert_{H^2}}{\sqrt{4\pi}} \left(\sum_{l=0}^\infty \frac{(2l+1)^2}{1+l(l+1)+l^2(l+1)^2}\right)^{1/2} < \frac{1}{2}\sqrt{\frac{1}{\pi}\left(\frac{\pi^2}{3}+2\right)}\Vert u \Vert_{H^2}.$$
\end{lemma}

\begin{proof}
We start writing $u = \sum_{l=0}^\infty \sum_{m=-l}^l a_{lm}Y_{lm}$. Since this series converges uniformly to $u$, we have
\begin{align}\label{laplacefourier}
\Vert u \Vert_{C^0} &\leq \sum_{l=0}^\infty \sum_{m=-l}^l \vert a_{lm} \vert \Vert Y_{lm}\Vert_{C^0} \nonumber \\ & = \sum_{l=0}^\infty \sum_{m=-l}^l \vert a_{lm} \vert \sqrt{\frac{2l+1}{4\pi}}
\end{align}
Now note that
$$\Vert u \Vert_{H^2}^2 = \sum_{l=0}^\infty \sum_{m=-l}^l  a_{lm}^2 \Vert Y_{lm}\Vert_{H^2}^2,$$
and it is simple to check that $\Vert Y_{lm} \Vert_{H^2}^2 = 1 + l(l+1) + l^2(l+1)^2$. Then
$$\Vert u \Vert_{H^2}^2 = \sum_{l=0}^\infty \sum_{m=-l}^l a_{lm}^2 \left(1+l(l+1)+l^2(l+1)^2\right)$$
and returning to \eqref{laplacefourier}, we can use Cauchy-Schwarz inequality to obtain
\begin{align}\label{ineq:c0norm}
\Vert u \Vert_{C^0} &\leq \sum_{l=0}^\infty  \sum_{m=-l}^l \vert a_{lm} \vert \frac{\sqrt{1+l(l+1)+l^2(l+1)^2}}{\sqrt{1+l(l+1)+l^2(l+1)^2}} \sqrt{\frac{2l+1}{4\pi}} \nonumber \\ &\leq \left(\sum_{l=0}^\infty \sum_{m=-l}^l a^2_{lm}(1+l(l+1)+l^2(l+1)^2)\right)^{1/2} \left(\sum_{l=0}^\infty \sum_{m=-l}^l \frac{2l+1}{4\pi}\frac{1}{(1+l(l+1)+l^2(l+1)^2)} \right)^{1/2} \nonumber \\ & = \frac{\Vert u \Vert_{H^2}}{\sqrt{4\pi}} \left(\sum_{l=0}^\infty \frac{(2l+1)^2}{1+l(l+1)+l^2(l+1)^2}\right)^{1/2}.
\end{align}
To this end, we shall estimate $\left(\sum_{l=0}^\infty b_l\right)^{1/2}$, where $b_l = \frac{(2l+1)^2}{1+l(l+1)+l^2(l+1)^2}$. Note that the series $\sum_l b_l$ converges since $b_l < \left(\frac{2l+1}{l(l+1)}\right)^2$ for $l\geq 1$ and
$$\left(\frac{2l+1}{l(l+1)}\right)^2 =\left( \frac{1}{l} + \frac{1}{l+1}\right)^2 = \frac{1}{l^2} + \frac{2}{l(l+1)}+ \frac{1}{(l+1)^2}.$$
Thus,
\begin{align*}
\sum_{l=0}^\infty b_l &= 1 + \sum_{l=1}^\infty b_l \\ &< 1 + \sum_{l=1}^\infty \frac{1}{l^2} + 2\sum_{l=1}^\infty \frac{1}{l(l+1)} + \sum_{l=1}^\infty \frac{1}{(l+1)^2} \\ &= 1 + \frac{\pi^2}{6} + 2 \sum_{l=1}^\infty \left(\frac{1}{l} - \frac{1}{l+1}\right) + \sum_{k=2}^\infty \frac{1}{k^2} \\ &= 1 + \frac{\pi^2}{6} + 2 + \left(\frac{\pi^2}{6}-1\right) \\ &= \frac{\pi^2}{3} +2.
\end{align*}
From this and \eqref{ineq:c0norm}, we get
$$\Vert u \Vert_{C^0} \leq \frac{\Vert u \Vert_{H^2}}{\sqrt{4\pi}} \left(\sum_{l=0}^\infty \frac{(2l+1)^2}{1+l(l+1)+l^2(l+1)^2}\right)^{1/2} < \frac{1}{2}\sqrt{\frac{1}{\pi}\left(\frac{\pi^2}{3}+2\right)}\Vert u \Vert_{H^2}.$$
\end{proof}

Finally, using integration by parts and the Poincaré inequality, we estimate $C_P$.

\begin{lemma}\label{lemma:cp}
Let $u_0 \in H^2(S^2)$ be a function with zero average. Then
$$\Vert u_0 \Vert_{H^2}^2 \leq \frac{7}{4} \Vert \Delta u_0 \Vert_{L^2}^2.$$
\end{lemma}

\begin{proof}
Since $u_0$ has zero average, the Poincaré inequality yields
\begin{equation}\label{ineq:poincare}
\Vert u_0 \Vert_{L^2}^2 \leq \frac{1}{\lambda_1(g_0)} \Vert \nabla u_0 \Vert_{L^2}^2 = \frac{1}{2}\Vert \nabla u_0 \Vert_{L^2}^2,
\end{equation}
where $\lambda_1(g_0)=1/2$ is the first nonzero Laplacian eigenvalue for the round metric $g_0$ on $S^2$. Moreover, integration by parts together with Cauchy-Schwarz inequality give us
\begin{equation}
\int_{S^2}  \Vert \nabla u_0 \Vert^2 \ dA_{g_0} = -\int_{S^2} u_0\Delta u_0 \ dA_{g_0} \leq \Vert u_0 \Vert_{L^2} \Vert \Delta u_0 \Vert_{L^2}.
\end{equation}
Combining this with \eqref{ineq:poincare}, we obtain
$$\Vert \nabla u_0 \Vert_{L^2}^2 \leq \frac{1}{\sqrt{2}}\Vert \nabla u_0 \Vert_{L^2} \Vert \Delta u_0 \Vert_{L^2},$$
and then
\begin{equation}\label{ineq:byparts}
\Vert \nabla u_0 \Vert_{L^2} \leq \frac{1}{\sqrt{2}} \Vert \Delta u_0 \Vert_{L^2}.
\end{equation}
At last, from \eqref{ineq:poincare} and \eqref{ineq:byparts}, we obtain
\begin{align*}
\Vert u_0 \Vert_{H^2}^2 &= \Vert u_0 \Vert_{L^2}^2 + \Vert \nabla u_0 \Vert_{L^2}^2 + \Vert \Delta u_0 \Vert_{L^2}^2 \\  &= \frac{3}{2} \Vert \nabla u_0 \Vert_{L^2}^2 + \Vert \Delta u_0 \Vert_{L^2}^2 \\ &\leq \frac{3}{4} \Vert \Delta u_0 \Vert_{L^2}^2 + \Vert \Delta u_0 \Vert_{L^2}^2 = \frac{7}{4} \Vert \Delta u_0 \Vert_{L^2}^2.
\end{align*}
\end{proof}

\addcontentsline{toc}{section}{References}
\bibliographystyle{alpha}
\bibliography{biblio}
\end{document}